\documentclass[reqno]{amsart}
\usepackage{amssymb,setspace,graphicx}
\usepackage{ifpdf}
\ifpdf
\usepackage[hyperindex]{hyperref}%,pagebackref
\else
\expandafter \ifx \csname dvipdfm \endcsname \relax
\usepackage[hypertex,hyperindex]{hyperref}
\else
\usepackage[dvipdfm,hyperindex]{hyperref}
\fi
\fi
\allowdisplaybreaks[4]
\numberwithin{equation}{section}
\theoremstyle{plain}
\newtheorem{theorem}{Theorem}[section]

\theoremstyle{remark}
\newtheorem{remark}{Remark}[section]
\DeclareMathOperator{\td}{d}

\begin{document}

\title[Some inequalities for the trigamma function]
{Some inequalities for the trigamma function in terms of the digamma function}

\author[F. Qi]{Feng Qi}
\address[Qi]{Institute of Mathematics, Henan Polytechnic University, Jiaozuo City, Henan Province, 454010, China; College of Mathematics, Inner Mongolia University for Nationalities, Tongliao City, Inner Mongolia Autonomous Region, 028043, China; Department of Mathematics, College of Science, Tianjin Polytechnic University, Tianjin City, 300387, China}
\email{\href{mailto: F. Qi <qifeng618@gmail.com>}{qifeng618@gmail.com}, \href{mailto: F. Qi <qifeng618@hotmail.com>}{qifeng618@hotmail.com}, \href{mailto: F. Qi <qifeng618@qq.com>}{qifeng618@qq.com}}
\urladdr{\url{http://qifeng618.wordpress.com}}

\author[C. Mortici]{Cristinel Mortici}
\address[Mortici]{Department of Mathematics, Valahia University of T\^argovi\c{s}te, Bd. Unirii 18, 130082 T\^argovi\c{s}te, Romania; Academy of Romanian Scientists, Splaiul Independen \c{t}ei 54, 050094 Bucharest, Romania}
\email{\href{mailto: C. Mortici <cristinel.mortici@hotmail.com>}{cristinel.mortici@hotmail.com}, \href{mailto: C. Mortici <cmortici@valahia.ro>}{cmortici@valahia.ro}}
\urladdr{\url{http://www.cristinelmortici.ro}}

\begin{abstract}
In the paper, the authors establish three kinds of double inequalities for the trigamma function in terms of the exponential function to powers of the digamma function. These newly established inequalities extend some known results. The method in the paper utilizes some facts from the asymptotic theory and is a natural way to solve problems for approximating some quantities for large values of the variable.
\end{abstract}



\thanks{This paper was typeset using \AmS-\LaTeX}

\maketitle

\section{Motivations and main results}

The classical Euler gamma function may be defined for $\Re(z)>0$ by
\begin{equation*}
 \Gamma(z)= \int_{0}^{\infty}t^{z-1}e^{-t} \td t.
\end{equation*}
The logarithmic derivative of the gamma function $\Gamma(z)$ is denoted by
\begin{equation*}
\psi(z)=\frac{\td}{\td z}[ \ln \Gamma(z)]=\frac{\Gamma'(z)}{\Gamma(z)}
\end{equation*}
and called the digamma function.
The derivatives $\psi'(z)$ and $\psi''(z)$ are called the trigamma and tetragamma functions respectively. As a whole, the functions $\psi^{(k)}(z)$ for $k\in\{0\}\cup\mathbb{N}$ are called the polygamma functions.
These functions are widely used in theoretical and practical problems in all branches of mathematical science. Consequently, many mathematicians were preoccupied to establish new results about the gamma function, polygamma functions, and other related functions.

\subsection{The first main result}

In 2007, Alzer and Batir~\cite[Corollary]{alz} discovered that the double inequality
\begin{equation}\label{alzer-batir-ineq-aml}
 \sqrt{2\pi} \,x^x \exp \biggl[-x-\frac12\psi(x+\alpha)\biggr]< \Gamma(x) < \sqrt{2\pi} \,x^x \exp \biggl[-x-\frac12\psi(x+ \beta) \biggr]
\end{equation}
holds for $x>0$ if and only if $\alpha \ge\frac13$ and $\beta \le0$. For information on generalizations of results in the paper~\cite{alz}, please refer to~\cite{Convexity2CM.tex, CM-Tri-Gamma.tex} and references cited therein.
Motivated by the double inequality~\eqref{alzer-batir-ineq-aml}, Mortici~\cite{m1} proposed the asymptotic formula
\begin{equation*}
 \Gamma(x) \sim \sqrt{2\pi}e^{-b}(x+b) ^{x} \exp \biggl[-x-\frac12\psi(x+c) \biggr], \quad x \to\infty
\end{equation*}
and then determined the optimal values of parameters $b$, $c$ in such a way that this convergence is the fastest possible.
\par
In 2004, Batir~\cite[Theorem 2.1]{batir0} presented the double inequality
\begin{multline} \label{bb}
 \Gamma(c) \exp[\psi(x)e^{\psi(x)}-e^{\psi(x)}+1] \le \Gamma(x) \\*
 \le \Gamma(c) \exp \biggl\{\frac{6e^{\gamma}}{\pi^2} \bigl[\psi(x) e^{\psi(x)}-e^{\psi(x)}+1 \bigr] \biggr\},
\end{multline}
for every $x \ge c$, where $c=1.461 \dotsc$ is the unique positive zero of the digamma function $\psi$ and $\gamma=0.577 \dotsc$ is the Euler-Mascheroni constant. In 2010, Mortici~\cite{m2} discovered the asymptotic formula
\begin{equation*}
 \Gamma(x) \sim\frac{\sqrt{2\pi}}{e} \exp \bigl[\psi(x)e^{\psi(x)}-e^{\psi(x)}+1 \bigr], \quad x \to\infty.
\end{equation*}
In 2011,  the double inequality~\eqref{bb} was generalized by~\cite[Theorem~2]{notes-best-simple-open-jkms.tex} to a monotonicity property which reads that the function
\begin{equation*}
f_{s,t}(x)=\begin{cases}\displaystyle
\frac{g_{s,t}(x)}{[g'_{s,t}(x)-1]\exp[g'_{s,t}(x)]+1},&x\ne c\\
\dfrac1{g''_{s,t}(c)},&x=c
\end{cases}
\end{equation*}
on $x\in(-\alpha,\infty)$ is decreasing for $\lvert t-s\rvert<1$ and increasing for $\lvert t-s\rvert>1$, where $s$ and $t$ are real numbers, $\alpha=\min\{s,t\}$, $c\in(-\alpha,\infty)$, and
\begin{equation*}
g_{s,t}(x)=\begin{cases}\displaystyle
\frac1{t-s} \int_c^x\ln\biggl[\frac{\Gamma(u+t)}{\Gamma(u+s)} \frac{\Gamma(c+s)}{\Gamma(c+t)}\biggr]\td u,&s\ne t\\
\displaystyle
\int_c^x[\psi(u+s)-\psi(c+s)]\td u,&s=t
\end{cases}
\end{equation*}
on $x\in(-\alpha,\infty)$.
\par
In 2000, Elezovi\'c, Giordano, and Pe\v{c}ari\'c~\cite{ele} found the single-sided inequality
\begin{equation} \label{ele}
\psi'(x)<e^{-\psi(x)}, \quad x>0.
\end{equation}
This inequality is closely related to the monotonicity and convexity of the function
\begin{equation*}
\mathcal{Q}(x)=e^{\psi(x+1)}-x
\end{equation*}
on $(-1,\infty)$. See also~\cite{MIA-1729.tex} and plenty of references therein.
By the way, as a conjecture posed in~\cite[Remark~3.6]{Infinite-family-Digamma.tex}, the complete monotonicity of the function $\mathcal{Q}(x)$ on $(0,\infty)$ still keeps open. An infinitely differentiable function $f$ is said to be completely monotonic on an interval $I$ if it satisfies $(-1)^kf^{(k)}(x)\ge0$ on $I$ for all $k\ge0$. This class of functions has applications in the approximation theory, asymptotic analysis, probability, integral transforms, and the like. See~\cite[Chapter~14]{Cheney-Light-AMS-09}, \cite[Chapter~XIII]{mpf-1993}, \cite[Chapter~1]{Schilling-Song-Vondracek-2nd}, and~\cite[Chapter~IV]{widder}.
For more information on the functions
\begin{equation*}
Q'(x-1)=\psi'(x)e^{\psi(x)}-1 \quad \text{and}\quad Q''(x-1)=\big\{\psi''(x)+[\psi'(x)]^2\bigr\}e^{\psi(x)},
\end{equation*}
please see the papers~\cite{Yang-Fan-2008-Dec-simp.tex, Infinite-family-Digamma.tex, Guo-Qi-Srivasta-Unique.tex, UPB-1635.tex, x-4-q-di-tri-gamma-rcsm.tex, deg4tri-tetra-short.tex, Mortici-monoburn.tex, BAustMS-5984-RV.tex, notes-best-simple-cpaa.tex, AAM-Qi-09-PolyGamma.tex, SCM-2012-0142.tex, notes-best.tex, Bukac-Sevli-Gamma.tex, x-4-di-tri-gamma-p(x)-Slovaca.tex, x-4-di-tri-gamma-upper-lower-combined.tex}, the expository and survey articles~\cite{bounds-two-gammas.tex, Gautschi-Kershaw-TJANT.tex, Wendel-Gautschi-type-ineq-Banach.tex, Wendel2Elezovic.tex-JIA}, and a number of references cited therein.
In 2011, Batir~\cite[Theorem~2.7]{batir} obtained the double inequality
\begin{equation} \label{ab}
(x+a^{\ast})e^{-2\psi(x+1)}<\psi'(x+1) \le(x+b^{\ast})e^{-2\psi(x+1)}, \quad x>0,
\end{equation}
where the constants $a^{\ast}=\frac12$ and $b^{\ast}=\frac{\pi^2}{6e^{2\gamma}}=0.518 \dotsc$ are the best possible.
In other words, Batir~\cite{batir} proposed the approximation formula
\begin{equation} \label{a1}
\psi'(x+1) \sim(x+a)e^{-2\psi(x+1)},
\end{equation}
where $a$ is a constant. A numerical computation shows that the approximation~\eqref{a1} gives a better result when choosing $a=a^{\ast}$, rather than the value $a=b^{\ast}$. This fact is somewhat expected as Batir obtained~\eqref{ab} as a result of the decreasing monotonicity of the function
\begin{equation*}
 \theta(x)=\psi'(x+1)e^{2\psi(x+1)}-x
\end{equation*}
on $(0,\infty)$, with $\theta(0)=\frac{\pi^2}{6e^{2\gamma}}$ and $\theta(\infty)=\frac12$.
Hence, the function $\theta(x)$ becomes gradually closer to
$\theta(\infty)$, as $x$ approaches infinity.
The decreasing monotonicity of the function $\theta(x)$ may be rewritten and extended as that the function
\begin{equation*}%\label{rel-eq-f-3}
 \Theta(x)=\psi'(x)e^{2\psi(x)}-x+\frac12
\end{equation*}
is decreasing and positive, but not convex or concave, on $(0,\infty)$. In 2010, among other things, Qi and Guo~\cite[Corollary~2]{AAM-Qi-09-PolyGamma.tex} proved that the function
\begin{equation*}%\label{lambda-psi=x+1}
f_{p,q}(x)=e^{p\psi(x+1)}-q x
\end{equation*}
for $p \ne0$ and $q \in \mathbb{R}$ is strictly convex with respect to $x \in(-1,\infty)$ if and only if $p \ge1$ or $p<0$.
Since
\begin{equation}\label{mortici-20-eq}
\frac{\td e^{\alpha\psi(x)}}{\td x}= \alpha\psi'(x)e^{\alpha\psi(x)} \quad \text{and} \quad
\frac{\td^2e^{\alpha\psi(x)}}{\td x^2}
= \alpha \bigl\{\psi''(x)+ \alpha[\psi'(x)]^2 \bigr\}e^{\alpha\psi(x)},
\end{equation}
the above functions between~\eqref{ele} and~\eqref{mortici-20-eq} have something to do with the more general function
\begin{equation*}
f_{a,b,\alpha,\beta,\lambda}(x)=e^{a\psi(x+b)}+\alpha x^2+\beta x+\lambda
\end{equation*}
and its monotonicity and convexity, even its complete monotonicity.
\par
For obtaining accurate approximations of the type~\eqref{a1}, the value $a=\frac12$ should be used. However, approximations of the form
\begin{equation*}
\psi'(x+1) \sim[x+a(x)]e^{-2\psi(x+1)},
\end{equation*}
with $a(x) \to\frac12$ as $x$ tends to $\infty$, are much better. Furthermore, several experiments using some asymptotic expansions lead us to the claim that approximations of the type
\begin{equation*}
\psi'(x+1) \sim[x+a(x)] \exp \biggl[-2\psi(x+1)-\frac1{120x^4} \biggr]
\end{equation*}
are more accurate. Consequently, we obtain the following theorem.

\begin{theorem}\label{000-mortici-psy-thm1}
For $x \ge3$, the double inequality
\begin{multline} \label{aba}
[x+ \alpha(x)] \exp \biggl[-2\psi(x+1)-\frac1{120x^4} \biggr] \le\psi'(x+1) \\
 \le[x+ \beta(x)] \exp \biggl[-2\psi(x+1)-\frac1{120x^4} \biggr]
\end{multline}
is valid, where
\begin{equation*}
 \alpha(x)=\frac12+\frac1{90x^3}-\frac1{60x^4} \quad \text{and} \quad
 \beta(x)=\frac12+\frac1{90x^3}.
\end{equation*}
\end{theorem}

\subsection{The second main result}

In 2013, Guo and Qi~\cite[Lemma 2]{Yang-Fan-2008-Dec-simp.tex} proved that
\begin{equation}\label{e-1-t-1}
\psi'(x)<e^{1/x}-1, \quad x>0.
\end{equation}
This inequality has been generalized to the complete monotonicity of the function
\begin{equation}\label{exp-trigamma-dif}
e^{1/x}-\psi'(x)-1, \quad x>0
\end{equation}
and others. For detailed information, please refer to~\cite{property-psi.tex, property-psi-ii-Munich.tex, Bessel-ineq-Dgree-CM-Simp.tex}, \cite[Theorem~3.1]{QiBerg.tex}, \cite[Theorem~1.1]{simp-exp-degree-revised.tex}, \cite[Theorem~1.1]{simp-exp-degree-new.tex}, and closely related references therein.
In some of these references, it was obtained that
\begin{equation*}
 \lim_{x \to\infty} \bigl[e^{1/x}-\psi'(x) \bigr]=1.
\end{equation*}
The approximation $\psi'(x) \sim e^{1/x}-1$ indeed gives a good result for the large value of $x$.
Now we propose the improvement
\begin{equation*}
\psi'(x) \sim e^{1/x+ \mu(x)}-1,
\end{equation*}
where $\mu(x) \to0$ and $\mu(x)=o(x)$ as $x \to\infty$. Our main result may be stated in details as the following theorem.

\begin{theorem}\label{000-mortici-psy-thm2}
For $x \ge3$, we have
\begin{equation}\label{000-mortici-psy-ineq}
e^{m(x)}-1<\psi'(x)<e^{M(x)}-1,
\end{equation}
where
\begin{equation*}
m(x)=\frac1x-\frac1{24x^4}+\frac{7}{360x^6} \quad \text{and} \quad M(x)=m(x) +\frac1{90x^7}.
\end{equation*}
\end{theorem}

\subsection{The third main result}

In 2014, Yang, Chu, and Tao found~\cite[Theorem~1]{Yang-Chu-Tao-AAA-14} that the constants $p=1$ and $q=2$ are the best possible real parameters such that the double inequality
\begin{equation} \label{teta}
 \theta(x,p)<\psi'(x+1)< \theta(x,q)
\end{equation}
holds for $x>0$, where
\begin{equation*}
 \theta(x,m)=\frac{e^{m/(x+1)}-e^{-m/x}}{2m}.
\end{equation*}
Because
\begin{equation}\label{psi'(x+1)recur}
\psi'(x+1)=\psi'(x)-\frac1{x^2},
\end{equation}
the double inequality~\eqref{teta} may be reformulated as
\begin{equation}\label{theta-x-sq-eq}
\frac1{x^2}+\theta(x,1)<\psi'(x)<\frac1{x^2}+\theta(x,2), \quad x>0.
\end{equation}
The complete monotonicity of the function~\eqref{exp-trigamma-dif} implies that
\begin{equation*}%\label{alpha-exp=psi-eq-bounds}
e-\psi'(1)> e^{1/(x+1)}-\psi'(x+1)>1, \quad x>0
\end{equation*}
which may be rearranged as
\begin{equation}\label{x+1-eq}
e^{1/(x+1)}-e+\psi'(1)<\psi'(x+1)<e^{1/(x+1)}-1<\frac12\sinh\frac2{x}
\end{equation}
on $(0,\infty)$. When $t<1.6\dotsc$, the lower bound $e^{1/(x+1)}-e+\psi'(1)$ is better than the corresponding one in~\cite[Corollary~2]{Yang-Chu-Tao-AAA-14}. The upper bound $e^{1/(x+1)}-1$ in~\eqref{x+1-eq} and the upper bound $\theta(x,2)$ in~\eqref{teta} can not be compared with each other. However, the upper bound $\frac1{x^2}+\theta(x,2)$ in~\eqref{theta-x-sq-eq} is better than the upper bound $e^{1/x}-1$ in~\eqref{e-1-t-1}.
\par
In this paper, we will improve the double inequality~\eqref{teta} as follows.

\begin{theorem}\label{000-mortici-psy-thm3}
For $x\ge1$, we have
\begin{equation}\label{z1}
\theta(x,1) +\frac1{24x^5}-\frac{5}{48x^6}<\psi'(x+1)< \theta(x,1) +\frac1{24x^5}
\end{equation}
and
\begin{equation}\label{z2}
 \theta(x,2) -\frac1{45x^7}<\psi'(x+1)
< \theta(x,2) -\frac1{45x^7}+\frac{7}{90x^8}.
\end{equation}
\end{theorem}

\section{Proofs of Theorems~\ref{000-mortici-psy-thm1} to~\ref{000-mortici-psy-thm3}}

Now we start out to prove our three main results.
\par
In 1997, Alzer~\cite[Theorem 8]{zzz0} proved that for $m,n \ge1$ the functions
\begin{equation*}%\label{fm}
F_{m}(x)= \ln \Gamma(x+1) -\biggl(x+\frac12\biggr) \ln x+x-\frac12 \ln2\pi- \sum_{i=1}^{2m}\frac{B_{2i}}{2i(2i-1) x^{2i-1}}
\end{equation*}
and
\begin{equation*}%\label{gn}
G_{n}(x)=- \ln \Gamma(x+1) +\biggl(x+\frac12\biggr) \ln x-x+\frac12 \ln2\pi+ \sum_{i=1}^{2n-1}\frac{B_{2i}}{2i(2i-1) x^{2i-1}}
\end{equation*}
are completely monotonic on $(0,\infty)$, where $B_j$ are the Bernoulli numbers given by the generating function
\begin{equation*}
\frac{t}{e^{t}-1}= \sum_{j=0}^{\infty}B_j\frac{t^{j}}{j!}
=1-\frac{t}2+\sum_{j=0}^{\infty}B_{2j}\frac{t^{2j}}{(2j)!}.
\end{equation*}
See also~\cite[Theorem~2]{Koumandos-jmaa-06}, \cite[Theorem~2.1]{Koumandos-Pedersen-09-JMAA}, and~\cite[Theorem~3.1]{Mortici-AA-10-134}. From $F_1'(x)<0$, $G_2'(x)<0$, $F_1''(x)>0$, $F_2''(x)>0$, $G_2''(x)>0$, and $G_3''(x)>0$, it follows that
\begin{equation}\label{22}
\ln x+\frac1{2x}-\frac1{12x^2}+\frac1{240x^4}-\frac1{252x^6}<\psi(x+1)
<\ln x+\frac1{2x}-\frac1{12x^2}+\frac{1}{240x^4},
\end{equation}
\begin{equation}\label{psi'(x+1)-l-power5}
\psi'(x+1) >\frac1x-\frac1{2x^2}+\frac1{6x^3}-\frac1{30x^5},
\end{equation}
\begin{multline}\label{33}
\frac1x-\frac1{2x^2}+\frac1{6x^3}-\frac1{30x^5}+\frac1{42x^7}-\frac1{30x^9}<\psi'(x+1)\\
<\frac1x-\frac1{2x^2}+\frac1{6x^3}-\frac1{30x^5}+\frac1{42x^7},
\end{multline}
and, also by~\eqref{psi'(x+1)recur},
\begin{multline}\label{power11-9-ineq}
\frac1x+\frac1{2x^2}+\frac1{6x^3}-\frac1{30x^5}+\frac1{42x^7}-\frac1{30x^9}<\psi'(x) \\
<\frac1x+\frac1{2x^2}+\frac1{6x^3}-\frac1{30x^5} +\frac1{42x^7}-\frac1{30x^9}+\frac{5}{66x^{11}}.
 \end{multline}

\begin{proof}[Proof of Theorem~\ref{000-mortici-psy-thm1}]
By taking the logarithm, the double inequality~\eqref{aba} may be rearranged as
\begin{equation*}
F(x)= \ln[x+ \alpha(x)]-2\psi(x+1) - \ln\psi'(x+1) -\frac1{120x^4}<0
\end{equation*}
and
\begin{equation*}
G(x)=- \ln[x+ \beta(x)] + \ln\psi'(x+1) +2\psi(x+1) +\frac1{120x^4}<0
\end{equation*}
for $x \ge3$. By virtue of inequalities~\eqref{22} and~\eqref{33}, one may deduce that $F(x)<F_1(x)$ and $G(x)<G_1(x)$, where
\begin{align*}
F_1(x)  &= \ln[x+ \alpha(x)] -2\biggl(\ln x+\frac1{2x}-\frac1{12x^2}+\frac1{240x^4} -\frac1{252x^6}\biggr) \\
&\quad - \ln\biggl(\frac1x-\frac1{2x^2}+\frac1{6x^3}-\frac1{30x^5} +\frac1{42x^7}-\frac1{30x^9}\biggr) -\frac1{120x^4}
\end{align*}
and
 \begin{align*}
G_1(x)  &=- \ln[x+ \beta(x)] + \ln\biggl(\frac1x-\frac1{2x^2} +\frac1{6x^3}-\frac1{30x^5}+\frac1{42x^7}\biggr) \\
&\quad +2\biggl( \ln x+\frac1{2x}-\frac1{12x^2}+\frac1{240x^4}\biggr)
+\frac1{120x^4}.
 \end{align*}
Since
\begin{equation*}
F_1'(x)=\frac{A(x-3)}{105x^7(180x^5+90x^4+2x-3) B(x-3)}
\end{equation*}
and
\begin{equation*}
G_1'(x)=\frac{C(x-3)}{15x^5(90x^4+45x^3+1) D(x-3)},
\end{equation*}
where
\begin{align*}
A(x)&=84000 x^{13}+3753750 x^{12}+75109720 x^{11}+896904120 x^{10}\\
&\quad+7160223140 x^9+40457327085 x^8+166700796732 x^7+507517074474 x^6\\
&\quad+1141703970759 x^5+1873185114060 x^4+2175691772642 x^3\\
&\quad+1690072075536 x^2+783944661553 x+162974708124,\\
B(x)&=210 x^8+4935 x^7+50750 x^6+298305 x^5+1096193 x^4\\
&\quad+2578821 x^3+3792857 x^2+3188649 x+1173161,\\
C(x)&=23625 x^9+604200 x^8+6818475 x^7+44510545 x^6\\
&\quad+184933335 x^5+506070905 x^4+909421185 x^3\\
&\quad+1030441127 x^2+663679092 x+183168418,\\
D(x)&=210 x^6+3675 x^5+26810 x^4+104370 x^3+228683 x^2+267393 x+130352
\end{align*}
are polynomials with positive coefficients, the functions $F_1(x)$ and $G_1(x)$ are strictly increasing on $[3,\infty)$. Furthermore, since
\begin{equation*}
\lim_{x \to\infty}F_1(x)= \lim_{x \to\infty}G_1(x)=0,
\end{equation*}
it follows that $F_1(x)<0$ and $G_1(x)<0$ for $x\ge3$. As a result, we have $F(x)<F_1(x)<0$ and $G(x)<G_1(x)<0$. The proof of Theorem~\ref{000-mortici-psy-thm2} is complete.
 \end{proof}

\begin{proof}[Proof of Theorem~\ref{000-mortici-psy-thm2}]
The inequality~\eqref{000-mortici-psy-ineq} can be written as
\begin{equation*}
m(x)< \ln[1+\psi'(x)]<M(x).
\end{equation*}
Considering the double inequality~\eqref{power11-9-ineq}, since $m(x) - \ln[1+\psi'(x)]<m_1(x)$ and $M(x) - \ln[1+\psi'(x)] >M_1(x)$, one may see that it suffices to show that
\begin{equation*}
m_1(x)\triangleq m(x) - \ln\biggl(1+\frac1x+\frac1{2x^2}+\frac1{6x^3} -\frac1{30x^5}+\frac1{42x^7}-\frac1{30x^9}\biggr)<0
\end{equation*}
and
\begin{equation*}
M_1(x)\triangleq M(x) - \ln\biggl(1+\frac1x+\frac1{2x^2}+\frac1{6x^3}-\frac1{30x^5} +\frac1{42x^7}-\frac1{30x^9}+\frac{5}{66x^{11}}\biggr)>0.
\end{equation*}
A straightforward computation gives
\begin{equation*}
m_1'(x)=\frac{E(x-3)}{60x^7G(x-3)}\quad \text{and}\quad M_1'(x)=-\frac{F(x-3)}{180x^8H(x-3)},
\end{equation*}
where
\begin{align*}
E(x)&=980 x^8+22485 x^7+221130 x^6+1212855 x^5+4032099 x^4\\
&\quad+8229303 x^3+9865371 x^2+6074127 x+1290163,\\
F(x)&=66495 x^{10}+2146155 x^9+31007240 x^8+263913573 x^7\\
&\quad+1464790565 x^6+5537745108 x^5+14437981040 x^4\\
&\quad+25626153678 x^3+29624987873 x^2+20135221233 x+6106987838,\\
G(x)&=210 x^9+5880 x^8+73185 x^7+531440 x^6+2481255 x^5+7724423 x^4\\
&\quad+16033731 x^3+21398207 x^2+16660569 x+5765861,\\
H(x)&=2310 x^{11}+78540 x^{10}+1213905 x^9+11258170 x^8+69614160 x^7\\
&\quad+301344043 x^6+931827204 x^5+2058324400 x^4+3182887290 x^3\\
&\quad+3281444518 x^2+2029943157 x+570820414
\end{align*}
are polynomials with positive coefficients. This means that $m_1'(x)>0$ and $M_1'(x)<0$ for $x\ge3$, that is, the function $m_1(x)$ is strictly increasing and the function $M_1(x)$ is strictly decreasing on $[3,\infty)$, with the limits
\begin{equation*}
\lim_{x \to\infty}m_1(x)= \lim_{x \to\infty}M_1(x)=0.
\end{equation*}
Accordingly, one obtain that $m_1(x)<0$ and $M_1(x)>0$ on $[3,\infty)$. The proof of Theorem~\ref{000-mortici-psy-thm2} is complete.
\end{proof}

\begin{proof}[Proof of Theorem~\ref{000-mortici-psy-thm3}]
The left hand side inequality in~\eqref{z1} can be written as
\begin{equation*}
c(x)\triangleq\frac1{2}e^{1/(x+1)}-\frac1{2}e^{-1/x}+\frac1{24x^5}-\frac{5}{48x^6}-\psi'(x+1)<0.
\end{equation*}
Since
\begin{equation*}
e^{-1/x}>\sum_{k=0}^7\frac{1}{k!}\biggl(-\frac1{x}\biggr)^k
\end{equation*}
and the inequality~\eqref{psi'(x+1)-l-power5} is valid, we have
\begin{equation*}
2c(x)<e^{1/(x+1)}-\frac{P(x)}{5040x^7},
\end{equation*}
where
\begin{equation*}
P(x)=5040 x^7+5040 x^6-2520 x^5+840 x^4+210 x^3-798 x^2+1057 x-1.
\end{equation*}
Hence, it is sufficient to prove that
\begin{equation*}
c_1(x)=\frac1{x+1}- \ln\frac{P(x)}{5040x^7}<0.
\end{equation*}
A straightforward differentiation yields
\begin{equation*}
c_1'(x)=\frac{7630 x^2+6329 x-7}{x(x+1)^2P(x)}=\frac{7630(x-1)^2+21589(x-1)+13952}{x(x+1)^2Q(x-1)},
\end{equation*}
where
\begin{align*}
Q(x)&=5040 x^7+40320 x^6+133560 x^5+240240 x^4\\
&\quad+255570 x^3+161112 x^2+56371 x+8868.
\end{align*}
This implies that the function $c_1(x)$ is strictly increasing on $[1,\infty)$. Furthermore, because of $\lim_{x \to\infty}c_1(x)=0$, it follows that $c_1(x)<0$ on $[1,\infty)$. The proof of the left hand side inequality in~\eqref{z1} is complete.
\par
Similarly, we may verify other inequalities in~\eqref{z1} and~\eqref{z2}. For the sake of saving the space and shortening the length of this paper, we do not repeat the processes. The proof of Theorem~\ref{000-mortici-psy-thm3} is complete.
\end{proof}

\section{Remarks}

Finally we give several remarks on our three main results.

\begin{remark}
The right hand side in~\eqref{aba} is better than the right hand side in~\eqref{ab}, but the left hand side in~\eqref{aba} is not better than the left hand side in~\eqref{ab}. This is because the inequalities
\begin{equation*}
(x+a^{\ast})e^{-2\psi(x+1)}>[x+ \alpha(x)] \exp \biggl[-2\psi(x+1)-\frac1{120x^4} \biggr]
\end{equation*}
and
\begin{equation*}
[x+ \beta(x)] \exp \biggl[-2\psi(x+1)-\frac1{120x^4} \biggr]<(x+b^{\ast})e^{-2\psi(x+1)},
\end{equation*}
which are equivalent to
\begin{equation}\label{01}
u(x)= \ln \biggl(x+\frac12+\frac1{90x^3}-\frac1{60x^4} \biggr) - \ln \biggl(x+\frac12 \biggr)-\frac1{120x^4}<0
\end{equation}
and
\begin{equation}\label{02}
v(x)= \ln \biggl(x+\frac12+\frac1{90x^3} \biggr)
-\frac1{120x^4}- \ln \biggl(x+\frac{\pi^2}{6e^{2\gamma}} \biggr)<0,
\end{equation}
are valid for $x \ge1$. The inequalities~\eqref{01} and~\eqref{02} may be verified as follows. A direct differentiation gives
 \begin{equation*}
u'(x)=\frac{540 x^6+60 x^5+180 x^4+90x \bigl(x-\frac1{30} \bigr)^2+\frac{9}{10}x+2} {30x^5(2x+1)S\bigl(x-\frac1{10}\bigr)}
>0
 \end{equation*}
for $x \ge\frac1{10}$ and
\begin{equation*}
v'(x)=\frac{Q(x-1)}{30x^5(90x^4+45x^3+1) (6e^{2\gamma}x+\pi^2)}>0
\end{equation*}
for $x \ge1$, where
\begin{equation*}
S(x)=180x^4+162x^3 +\frac{189}{5}x^2 +\frac{21}{50}x+\frac{226}{125}
\end{equation*}
and
\begin{align*}
Q(x)&=2700 \bigl(\pi^2-3e^{2\gamma} \bigr)x^8+21600 \bigl(\pi^2-3e^{2\gamma} \bigr)x^7+75600 \bigl(\pi^2-3e^{2\gamma} \bigr)x^6 \\
& \quad+180 \bigl(840\pi^2-2521e^{2\gamma} \bigr)x^5+630 \bigl(300\pi^2-901e^{2\gamma} \bigr)x^4 \\
& \quad+45 \bigl(3361\pi^2-10096e^{2\gamma} \bigr)x^3+45 \bigl(1683\pi^2-5044e^{2\gamma} \bigr)x^2 \\
& \quad+3 \bigl(7245\pi^2-21538e^{2\gamma} \bigr)x+2 \bigl(1373\pi^2-4002e^{2\gamma} \bigr)
 \end{align*}
are polynomials with positive coefficients. This means that the functions $u(x)$ and $v(x)$ are strictly increasing on $[1,\infty)$. Further, from the limits
\begin{equation*}
 \lim_{x \to\infty}u(x)= \lim_{x \to\infty}v(x)=0,
\end{equation*}
it follows that $u(x)<0$ on $(0,\infty)$ and $v(x)<0$ on $[1,\infty)$.
\end{remark}

\begin{remark}[The asymptotic series of $\psi'(x+1)e^{2\psi(x+1)}$]
Whenever an approximation formula $f(x) \sim g(x)$ is considered in the sense that the ratio $\lim_{x \to\infty}\frac{f(x)}{g(x)}=1$ (sometimes $\lim_{x\to\infty}[f(x) -g(x)]=0$), there is a tendency to improve it by adding new terms, or an entire series, of the form
\begin{equation*}
f(x) \sim g(x) + \sum_{k=1}^{\infty}\frac{a_k
}{x^k}.
\end{equation*}
Such a series is called an asymptotic series and it plays a central role in the theory of approximation. Although such a series is often divergent, by truncating at the $m$-th term, it provides approximations of any desired accuracy $\frac1{x^{m+1}}$. In other words, the formula
\begin{equation*}
f(x)\sim g(x) + \sum_{k=1}^{m}\frac{a_k}{x^k
}+O\biggl(\frac1{x^{m+1}}\biggr), \quad x\to\infty
\end{equation*}
is valid for every integer $m \ge1$.
\par
It is well known~\cite{ab} that the digamma and trigamma functions admit respectively the following asymptotic expansions:
\begin{equation}\label{0}
\psi(x+1)\sim \ln x+\frac1{2x}- \sum_{k=2}^{\infty}\frac{B_k}{kx^k}\quad \text{and}\quad
\psi'(x+1)\sim -\frac1{x^2}+ \sum_{k=1}^{\infty}\frac{B_{k-1}}{x^k}.
\end{equation}
The asymptotic expansions in~\eqref{0} can be written explicitly as
 \begin{equation*}
\psi(x+1)\sim\ln x+\frac1{2x}-\frac1{12x^2}+\frac{1}{240x^4}-\frac1{252x^6} +\frac1{240x^8}-\frac1{132x^{10}}+ \dotsm
 \end{equation*}
and
 \begin{equation*}
\psi'(x+1)\sim\frac1x-\frac1{2x^2}+\frac{1}{6x^3}-\frac1{30x^5}+\frac1{42x^7} -\frac1{30x^9}+\frac{5}{66x^{11}}-\dotsm.
 \end{equation*}
\par
In order to construct the asymptotic expansion of the function $\psi'(x+1)e^{2\psi(x+1)}$, we recall from the papers~\cite{Chen-Elezovic-Vuksic-JCA, Gould-AMS-1974} and the monographs~\cite[pp.~20--21]{Erdelyi-1956} and~\cite[pp.~539\nobreakdash--541]{Knopp} the following classical results from the theory of asymptotic series:
\begin{enumerate}
\item
if
\begin{equation*}
u(x)\sim \sum_{k=0}^{\infty}\frac{p_k}{x^k} \quad \text{and}\quad
v(x)\sim \sum_{k=0}^{\infty}\frac{q_k}{x^k}
\end{equation*}
as $x\to\infty$, then
\begin{equation}\label{mortici-prop-2}
u(x)v(x)\sim \sum_{k=0}^{\infty}\frac{r_k}{x^k},
\end{equation}
where
\begin{equation*}
r_k= \sum_{i+j=k}p_{i}q_j;
\end{equation*}
\item
if
\begin{equation*}
f(x)\sim\sum_{k=0}^{\infty}\frac{a_k}{x^k}, \quad x\to\infty,
\end{equation*}
then
\begin{equation}\label{mortici-prop-1}
 \exp f(x)\sim\sum_{k=0}^{\infty}\frac{\alpha_k}{x^k}, \quad  x \to\infty,
\end{equation}
where $\alpha_{0}= \exp a_{0}$ and
\begin{equation*}
 \alpha_k= \sum_{j=1}^k\frac1{j!} \sum_{i_1+ \dotsm+i_j=k}\prod_{\ell=1}^ja_{i_\ell}
\end{equation*}
for $k \ge1$.
\end{enumerate}
\par
Now the asymptotic series of the function $\psi'(x+1)e^{2\psi(x+1)}$ can be computed in two steps. Firstly, by virtue of the formula~\eqref{mortici-prop-1}, we may transform the series~\eqref{0} to obtain the series of $e^{2\psi(x+1)}$. Secondly, with the help of the formula~\eqref{mortici-prop-2}, we multiply the series of $e^{2\psi(x+1)}$ and the second series in~\eqref{0}.
\par
Because the general term (in terms of the Bernoulli numbers) of the series of the function $\psi'(x+1)e^{2\psi(x+1)}$ has an unattractive form, we just write down the first few terms as follows:
 \begin{equation}\label{sss}
\psi'(x+1)e^{2\psi(x+1)} =x+\frac{1}{2}+\frac1{90x^3}-\frac1{60x^4}+\frac{2}{567x^5} +\frac{43}{2268x^6}-O\biggl(\frac1{x^7}\biggr).
 \end{equation}
Till now we can see immediately that the functions $\alpha(x)$ and $\beta(x)$ in Theorem~\ref{000-mortici-psy-thm1} are truncations of the series~\eqref{sss} at the first three or four terms.
\par
When more terms of~\eqref{sss} are considered, the conclusion of Theorem~\ref{000-mortici-psy-thm1} remains true and the estimates become more accurate.
\end{remark}

\begin{remark}
We conjecture that the function $e^{M(x)}-\psi'(x)-1$ is completely monotonic on $(0,\infty)$, but the function $\psi'(x)-e^{m(x)}+1$ is not.
\end{remark}

\begin{remark}
The reason why the double inequality~\eqref{teta} is of interest is because it provides much accurate estimates for the trigamma function $\psi'(x)$ for large values of $x$, that is,
\begin{equation}\label{mm}
\psi'(x+1) \sim \theta(x,m), \quad x \to\infty.
\end{equation}
From~\eqref{z1} and~\eqref{z2}, it is easily deduced that
\begin{equation*}
\psi'(x+1) - \theta(x,1)=O\biggl(\frac1{x^5}\biggr)\quad \text{and}\quad \psi'(x+1) - \theta(x,2)=O\biggl(\frac1{x^7}\biggr),
\end{equation*}
then it is easy to see that the best approximation of the form~\eqref{mm} may be obtained by taking $m=2$.
\end{remark}

\begin{remark}
Evidently, our inequalities~\eqref{z1} and~\eqref{z2} are much stronger than those in~\cite{Yang-Chu-Tao-AAA-14}. Our proofs for~\eqref{z1} and~\eqref{z2} are also much simpler than the original proof of~\cite[Theorem~1]{Yang-Chu-Tao-AAA-14}. This shows that the natural approach to solve problems of approximating some quantities for large values of the variable is the theory of asymptotic series. This method or approach has been utilized and applied in the papers~\cite{m8, m6, m7, m3, m4, m5}, for examples.
\end{remark}

\begin{remark}
In this paper, we essentially talk about the relations between inequalities, asymptotic approximations, and complete monotonicity.
\end{remark}

\subsection*{Acknowledgements}
The first author was in part supported by the NNSF of China under Grant No.~11361038.
The second author was partially supported by a Grant of the Romanian National Authority for Scientific Research, CNCS-UEFISCDI, with the Project Number PN-II-ID-PCE-2011-3-0087.

 \end{document}